\documentclass[11pt]{article}
\usepackage [dvips]{graphics}
\usepackage[centertags]{amsmath}
\usepackage{amsfonts}
\usepackage{amssymb}
\usepackage{amsthm}
\usepackage{newlfont}
\usepackage{stmaryrd}

\theoremstyle{plain}
\newtheorem{thm}{Theorem}[section]
\newtheorem{cor}[thm]{Corollary}
\newtheorem{lem}[thm]{Lemma}
\newtheorem{prop}[thm]{Proposition}
\newtheorem{exam}[thm]{Example}
\newtheorem{rem}[thm]{Remark}
\newtheorem{defi}[thm]{Definition}

\def\sqr#1#2{{\vcenter{\vbox{\hrule height.#2pt
              \hbox{\vrule width.#2pt height#1pt \kern#1pt \vrule
width.#2pt}
              \hrule height.#2pt}}}}
\def\3n{\negthinspace \negthinspace \negthinspace }
\def\2n{\negthinspace \negthinspace }
\def\1n{\negthinspace }

\def\no{\noindent}

\def\ms{\medskip}
\def\bs{\bigskip}

\def\limsup{\mathop{\overline{\mathrm {lim}}}}
\def\liminf{\mathop{\underline{\mathrm {lim}}}}

\def\dim{\hbox{\rm dim$\,$}}

\def\span{\hbox{\rm span$\,$}}

\def\({\Big (}
\def\){\Big )}
\def\[{\Big[}
\def\]{\Big]}

\def\be{\begin{equation}}
\def\bel{\begin{equation}\label}
\def\ee{\end{equation}}
\def\bea{\begin{eqnarray}}
\def\eea{\end{eqnarray}}
\def\bt{\begin{theorem}}
\def\et{\end{theorem}}
\def\bc{\begin{corollary}}
\def\ec{\end{corollary}}
\def\bl{\begin{lemma}}
\def\el{\end{lemma}}
\def\bp{\begin{proposition}}
\def\ep{\end{proposition}}
\def\br{\begin{remark}}
\def\er{\end{remark}}
\def\ba{\begin{array}}
\def\ea{\end{array}}
\def\bd{\begin{definition}}
\def\ed{\end{definition}}

\newcommand{\poly}{\mathbb{C}[z_1,\ldots,z_d]}

\newcommand{\dif}{\mathrm{d}}

\newcommand{\Sing}{\mathrm{Sing}}
\newcommand{\Hardy}{H^2(\mathbb{D}^d)}

\newcommand{\ao}{asymptotically orthogonal }
\newcommand{\range}{\mathrm{ran~}}
\newcommand{\Ind}{\mathrm{Ind~}}
\newcommand{\coker}{\mathrm{coker~}}

\makeatletter
   
   \@addtoreset{equation}{section}
\makeatother

\begin{document}

\title{\bf Polydisc version of Arveson's conjecture}
\author{ Penghui Wang\thanks{Partially supported by NSFC(No. 11471189), E-mail: phwang@sdu.edu.cn }
\quad Chong Zhao \thanks{Partially supported by NSFC(No. 11501329) and Shandong Province Natural Science Foundation ZR2014AQ009, {
E-mail:} { chong.zhao@sdu.edu.cn}. \ms}
 \\ \\
 School of Mathematics, Shandong University\\
Jinan, Shandong 250100, The People's Republic of China\\
}

\maketitle
\begin{abstract}
In the present paper, we solve the polydisc-version of Arveson Conjecture by giving a complete criterion for essential normality  of homogeneous quotient modules of the Hardy module over the polydisc, and it turns  out that our method applies to quotient modules of the weighted Bergman modules $A^2_s(\mathbb D^d)$.
\end{abstract}

\bs

\no{\bf 2000 MSC}. 47A13, 46H25

\bs

\no{\bf Key Words}. Essential normality, Hardy space over the polydisc, quotient module
\section{Introduction}

~~~~Throughout this paper, $\mathbb{D}$ denotes  the open unit disc of the complex plane. The Hardy space $H^2(\mathbb{D}^d)$ over the polydisc is defined as the Hilbert space of analytic functions over $\mathbb{D}^d$ satisfying
$$||f||^2=\sup_{0<r<1}\int_{\mathbb{T}^d}|f(rz)|^2\dif m(z)<\infty,$$
where $\dif m$ is the normalized Haar measure on $\mathbb{T}^d$. $\Hardy$ can be viewed  as a Hilbert
module \cite{CG,DP} in the sense that $\Hardy$ as a Hilbert space admits a natural module structure over
the polynomial ring $\poly$  with respect to the obvious multiplication action.

A closed subspace $\mathcal{M}$ of $\Hardy$ that is invariant under multiplication by polynomials is called a submodule, and $\mathcal{N}=\Hardy\ominus\mathcal{M}$ which is invariant under the adjoint module actions is called a quotient module, whose module structure is given by
$$
 f\cdot g=S_f g, \quad f\in \poly,\, g\in \cal N,
$$
 where $S_{f}=P_{\mathcal{N}}M_f\mid_{\mathcal{N}}$. The closure of an ideal $I\subset\poly$ is a submodule, which is called the submodule generated by $I$ and denoted by $[I]$. For convenience we denote by $[I]^\bot$ the quotient module associate to the ideal $I$. If all the commutators $[S_{z_i}^*,S_{z_j}],1\leq i,j\leq d$ belong to the compact operator algebra $\mathcal{K}$, then $\mathcal{N}$ is said to be essentially normal.

The present paper is a continuation of \cite{WZ}, aiming at giving a complete answer to the polydisc version of Arveson's conjecture.
In \cite{Ar1}, Arveson conjectured that homogenous submodules of the $d$-shift module over the unit ball are essentially normal, and many efforts have been made along this line, such as \cite{Ar1,Ar3,Dou1,Dou2,DTY,DW,Eng,FX,Guo,GWk1,KS} and references therein.

The essential normality of quotient modules of canonical analytic function Hilbert modules over the polydiscs was initiated  in \cite{DM}, where some special modules such as $[(z-w)^2]^\perp$, $[z^i-w^j]^\perp$ in $H^2(\mathbb D^2)$ were considered. Clark\cite{Cla} proved that $[B_1(z_1)-B_2(z_2),\cdots,B_{d-1}(z_{d-1})-B_d(z_d)]^\perp$ can be identified with a kind of Bergman space over the associated varieties, where $B_i(z_i)$ are finite Blaschke products, and hence they are essentially normal. For recent development on essential normality over polydiscs, refer to \cite{GWk3,GWp1,GWp2,GWZ,KGS,Wa} and references therein.

Compared to the existing results on Arveson's conjecture over the unit ball, the situations over the polydisc are totally different. Briefly speaking, over the unit ball all the submodules generated by polynomials and their associated quotient modules are believed to be essentially normal, while over polydiscs no non-trivial submodule and few quotient modules are essentially normal.

For an ideal $I\subset \mathbb C[z_1,\cdots,z_d]$, denote by $Z(I)$ the zero variety of $I$, and $Z_{\mathbb D^d}(I)=Z(I)\cap \mathbb D^d$.
If $\dim_{\mathbb{C}}Z(I)=0$,  the quotient module $[I]^\perp$ is of finite dimension, on which the essential normality is trivial. Therefore in what follows, we always assume $\dim_{\mathbb{C}}Z(I)\geq 1$. It is mentioned in \cite{WZ} and will be proved in Section 2 that, if the homogenous quotient module $[I]^\perp$ is essentially normal, then $\dim_{\mathbb{C}}Z(I)=1$. Therefore by homogeneity, $Z_{\mathbb D^d}(I)=\bigcup_i V_i$ for several different discs $V_i$. Taking $u_i\in \partial V_i$ and we can write
$V_i=\{\lambda u_i:\lambda\in\mathbb{D}\}.$ To continue, we set
$$
\Lambda_i=\{i_j:\, |u_{i,i_j}|=1\}\subset\{1,2,\cdots,d\}.
$$
Obviously, $\Lambda_i$ depends only on $V_i$. In what follows, for $u\in \partial \mathbb D^d$ and $V_u=\{\lambda u: \lambda\in\mathbb D\}$, we will write $\Lambda=\{i_j:\, |u_{i_j}|=1\}$ if there is no ambiguity, and $u_\Lambda=(u_{i_1},\cdots,u_{i_k})$. Next, we introduce the following condition on the variety of $I$.
\vskip2mm
\noindent{\bf Condition A.} {\it Let $I$ be a homogenous ideal, $Z_{\mathbb D^d}(I)=\bigcup\limits_{i=1}^n V_i$, where $\{V_i\}$ are different discs. We say that $I$(or the zero variety $Z_{\mathbb D^d}(I)$) satisfies Condition A, if one of the following items holds
\begin{itemize}
\item[1)] $\Lambda_i\not=\Lambda_j$ for $i\not=j$;
\item[2)] For any pair $(i,j)$ such that $\Lambda_i=\Lambda_j$, let $\Lambda_i=\{i_1,\cdots, i_k\}$, then $(u_{i,i_1},\cdots,u_{i,i_k})$ and $(u_{j,i_1},\cdots,u_{j,i_k})$ are linearly independent.
\end{itemize}}
To state our main result, we need the following
\begin{defi}\label{def:basictype}
Let $u\in\partial\mathbb D^d$, and $I$ be a homogenous ideal such that $Z_{\mathbb{D}^d}(I)=V_u$. Denote by $J_u$ the ideal of $\poly$ generated by $\{\bar{u}_iz_i-\bar{u}_jz_j:i,j\in\Lambda\}$, and $I'$ the ideal generated by $I$ and $J_u$.
\begin{itemize}
\item[1)] If $I'=\sqrt{I}$ the prime ideal associated to $V_u$, then $I$ is called quasi-prime.
\item[2)] If $\sqrt{I}/I'$ is of finite dimension, then $I$ is called essentially quasi-prime.
 \end{itemize}The quotient module $[I]^\bot$ is called (essentially) quasi-prime provided $I$ is (essentially) quasi-prime.
\end{defi}
We can state our main result, which is a complete criterion for essential normality of homogeneous quotient modules of $\Hardy$.
\begin{thm}\label{thm:iff}
    Let $I\subset\poly$ be a homogeneous ideal for which $[I]^\bot$ is of infinite dimension. Then the quotient module $[I]^\bot$ of $\Hardy$ is essentially normal if and only if the following items hold,
    \begin{itemize}
     \item[1)] $Z_{\mathbb{D}^d}(I)$ satisfies Condition A, hence $Z_{\mathbb D^d}(I)=\bigcup\limits_i V_{u_i}$ for several different discs $V_{u_i}$,
     \item[2)] Let $I=\bigcap_{j=1}^m I_{u_j}$ be the primary decomposition with $Z_{\mathbb D^d}(I_{u_j})=V_{u_j}$, then each $I_{u_j}$ is essentially quasi-prime.
     \end{itemize}
\end{thm}
The criterion  in Theorem \ref{thm:iff}  is purely algebraic, which does not depend on the structure of Hardy module $\Hardy$. Notice that in the case $d=2$, Theorem \ref{thm:iff} reduces to \cite[Theorem 1.1]{GWp1}. In Section 4, we will give two surprising examples, which show that the essential normality in higher dimensional case is more interesting than the $2$-dimensional case.

\begin{rem}
\begin{itemize}
    \item[1)] Theorem \ref{thm:iff} suggests that the essential normality of quotient modules is closely  related to the distinguished variety\cite{AM}.
    \item[2)] The method in \cite{GWp1}, where d=2 is assumed, relies heavily on the fact that $M_{z_i}$ is isometric, and can not be applied to the weighted Bergman space.

\end{itemize}
\end{rem}

In \cite[Remark 2.15]{WZ}, we can only deal with the essential normality of quotient modules of the weighted Bergman module $A_s^2(\mathbb D^d)$ in some simplest cases, where $ A_s^2(\mathbb D^d)=A_s^2(\mathbb D)\otimes\cdots\otimes A_s^2(\mathbb D)$ and $A_s^2(\mathbb D)$ is the space of analytic functions $f$ on $\mathbb D$ such that
        $$
       \|f\|_s^2=\int_{\mathbb D}|f(z)|^2(1-|z|^2)^s\mathrm{d}A(z)<\infty,
        $$
    where $s>0$ and $\mathrm{d}A$ is the normalized area measure on $\mathbb D$. It can be verified  that  all the proofs in the present paper are valid to the weighted Bergman module case. We have
\begin{thm}\label{thm:iff-wbs}
For a homogenous ideal $I$ of $\poly$, the quotient module $[I]^\perp$ of $A_s^2(\mathbb D^d)$ is essentially normal if and only if 1)-2) in Theorem \ref{thm:iff} are satisfied.
\end{thm}

It is worth to point out that Theorem \ref{thm:iff-wbs} is the first nontrivial  result on essential normality of quotient modules of $A_s^2(\mathbb D^d)$.

\vskip2mm

The present paper is organized as following. in section 2 we give a geometric characterization of $Z(I)$ for essential normality of $[I]^\perp$. In Section 3 we construct a large class of essentially normal quotient modules of $\Hardy$. In Section 4 the complete criterion for essentially normal homogeneous quotient modules is obtained, and Section 5 contains some discussion on the K-homology.
\newpage
\section{The variety of an essentially normal quotient module}
We begin the present section with some preliminaries. Recall that $H^2(\mathbb{D}^d)$ is the Hilbert space with reproducing kernel
$$K_\lambda(z)=\prod_{i=1}^d(1-\bar{\lambda}_iz_i)^{-1},~\forall z\in\mathbb{D}^d$$
at $\lambda\in\mathbb{D}^d$. Obviously $||K_\lambda||^2=\prod_{i=1}^d(1-|\lambda_i|^2)^{-1}$ diverges to infinity as $\lambda$ approaches $\partial\mathbb{D}^d$.

For each $A\in B(H^2(\mathbb{D}^d))$ the function
$$\tilde{A}(z)=\langle Ak_z,k_z\rangle,z\in\mathbb{D}^d$$
is called the Berezin transform of $A$, where $k_z=||K_z||^{-1}K_z$ is the normalized reproducing kernel at $z$. It is routine to compute for $z,\lambda\in\mathbb{D}^d$ that
$$\langle k_z,K_\lambda\rangle=k_z(\lambda)=\frac{1}{||K_z||}\prod_{i=1}^d\frac{1}{1-\bar{\lambda}_iz_i},$$
which converges to $0$ as $z$ approaches $\partial\mathbb{D}^d$. Since linear combinations of $\{K_\lambda: \lambda\in\mathbb{D}^d\}$ are dense in $H^2(\mathbb{D}^d)$, $k_z$ converges to $0$ in the weak topology as $z$ approaches $\partial\mathbb{D}^d$. If $A$ is compact then as $z$ approaches $\partial\mathbb{D}^d$, $Ak_z$ converges to $0$ in norm and then $\tilde{A}(z)$ converges to $0$. Then we have the following well-known lemma.
\begin{lem}\label{lem:Berezin}
    It holds for each compact operator $A\in B(H^2(\mathbb{D}^d))$ that $\lim_{z\to\partial\mathbb{D}^d}\tilde{A}(z)=0$.
\end{lem}
By an application of Berezin transform, we get the following result.
\begin{lem}\label{lem:counterexample}
    Suppose $I\subset\poly$ is an ideal. Suppose $u,v\in\partial\mathbb{D}^d$ are two different accumulation points of $Z_{\mathbb D^d}(I)$, and there is a subset $\Lambda\subset\{1,\ldots,d\}$ such that
        $$u_i=v_i\in\mathbb{T},\forall i\in\Lambda,$$
    and $u_i,v_i\in\mathbb{D}$ whenever $i\in\Lambda^c$, then $[I]^\bot$ is not essentially normal.
\end{lem}
\begin{proof}
    Since $u\neq v$, there is a polynomial $h\in\poly$ such that $h(u)=0$ and $h(v)\neq0$. Let $\{u^{(n)}\},\{v^{(n)}\}$ be sequences in $Z_{\mathbb D^d}(I)$ with limits $u,v$ respectively. As $n\to\infty$, both $k_{u^{(n)}}$ and $k_{v^{(n)}}$ converge to $0$ in the weak topology. On the other hand,
    \begin{eqnarray*}
        &&\limsup_{n\to\infty}\langle[S_h^*,S_h]k_{u^{(n)}},k_{u^{(n)}}\rangle\\
        &=&\limsup_{n\to\infty}\langle S_hk_{u^{(n)}},S_hk_{u^{(n)}}\rangle-\lim_{n\to\infty}\langle S_h^*k_{u^{(n)}},S_h^*k_{u^{(n)}}\rangle\\
        &\geq&\limsup_{n\to\infty}|\langle S_hk_{u^{(n)}},k_{v^{(n)}}\rangle|^2-\lim_{n\to\infty}|h(u^{(n)})|^2\\
        &=&\lim_{n\to\infty}|h(v^{(n)})\langle k_{u^{(n)}},k_{v^{(n)}}\rangle|^2-|h(u)|^2\\
        &=&|h(v)|^2\lim_{n\to\infty}|\langle k_{u^{(n)}},k_{v^{(n)}}\rangle|^2\\
        &=&|h(v)|^2\prod_{i\in\Lambda^c}\frac{\sqrt{(1-|u_i|^2)(1-|v_i|^2)}}{1-\bar{u}_iv_i}\\
        &>&0.
    \end{eqnarray*}
    Therefore $\{[S_h^*,S_h]k_{u^{(n)}}\}$ does not converge to $0$ as $n\to\infty$, and $[S_h^*,S_h]$ cannot be compact by Lemma \ref{lem:Berezin}.
\end{proof}
\begin{exam}
    In the case $d=3$, let $I\subseteq C[z_1,z_2]$ be a homogenous ideal such that $Z(I)\cap\partial\mathbb D^2\subset \mathbb T^2$, and $p(z)=(\alpha_1 z_1-z_3)(\alpha_2 z_2-z_3)$ with $\alpha_1\not=\alpha_2$, $|\alpha_i|<1$. Then the quotient module $[I,p]^\perp$ of $H^2(\mathbb{D}^3)$ cannot be essentially normal according to Lemma \ref{lem:counterexample}.
\end{exam}
The following lemma and its idea of proof originate from \cite{WZ}, and we perform the improved proof here. It characterizes the zero varieties of homogeneous ideals $I$ for witch $[I]^\bot$ is essentially normal.
\begin{lem}\label{prop:varietydim}
    Let $I\subset\poly$ be a homogeneous ideal such that $\dim_{\mathbb{C}}Z(I)\geq2$, then $\mathcal{N}=[I]^\bot$ is not essentially normal.
\end{lem}
\begin{proof}
    Suppose $\dim_{\mathbb{C}}Z(I)=k\geq2$. Let $V$ be any irreducible component of $Z(I)$ of complex dimension $k$. For $i=1,\ldots,d$, denote by
    $$E_i=\{z\in V:|z_i|=\max_{1\leq j\leq d}|z_j|\}.$$
    Clearly each $E_i$ is closed as a subset of $V$, and $V=\bigcup_{1\leq i\leq d}E_i$. Then by the Baire Category Theorem, some $E_i$ is of the second category as a subset of $V$, and therefore we can find an open subset $\Omega_1\subset E_i$. Then since $\Sing (V)$, the set of singular points of $V$, is of the first category, there must be a nonsingular point $u\in\Omega_1\subset E_i$. By homogeneity we may assume $u\in\partial\mathbb{D}^d$. Let $g_1,\ldots,g_m$ be a set of generators for the prime ideal of $V$, then the matrix $\left(\frac{\partial g_j}{\partial z_i}(u)\right)_{i,j}$ is of rank $d-k$. By the Implicit Function Theorem, there is an open ball $B\subset\mathbb{C}^k$, an open neighbourhood $\Omega_2\subset\Omega_1$ of $u$, and an analytic bijection $\varphi$ that maps $B$ onto $\Omega_2$.

    Define the maps $\psi_j:\Omega_2\to\mathbb{C}^d,z\mapsto z_j/z_i,1\leq j\leq d$, and set $\psi=(\psi_1,\ldots,\psi_d)$. Since $\Omega_2\subset E_i$, $\psi_j$ maps $\Omega_2$ into $\bar{\mathbb{D}}$, and $\psi$ maps $\Omega_2$ into $Z(I)\cap\partial\mathbb{D}^d$. Then $\psi_j\circ\varphi$ maps $B$ analytically into $\bar{\mathbb{D}}$. Therefore $\psi_j\circ\varphi$ maps $B$ either onto an open subset of $\mathbb{D}$ or to a constant $c_j$ in $\bar{\mathbb{D}}$. If $\psi$ is constant on $\Omega_2$, then $\Omega_2\subset\{z\in\mathbb{C}^d:z_j=c_jz_i,1\leq j\leq d\}$, which cannot be of dimension $k\geq2$. Therefore some $\psi_j$ is nonconstant on $\Omega_2$, and hence there exists $v\in\Omega_2$ such that $\psi_j(v),\psi_j(u)$ are two different points in $\mathbb{D}$. Clearly for $1\leq k\leq d$, $\psi_k(u)$ and $\psi_k(v)$ should either both be in $\mathbb{D}$, or be of the same value in $\mathbb{T}$. Then since $\psi(u),\psi(v)\in Z(I)\cap\partial\mathbb{D}^d$, $[I]^\bot$ cannot be essentially normal by Lemma \ref{lem:counterexample}.
\end{proof}

\begin{rem}
Lemma \ref{prop:varietydim} shows that, in the polydisc case, if the homogenous quotient module is essentially normal, then the dimension of zero variety is less than or equal to 1. We speculate that in Lemma \ref{prop:varietydim} the requirement of homogeneity could be dropped.
\end{rem}

Combine Lemma \ref{lem:counterexample} and Lemma \ref{prop:varietydim} and we immediately obtain the following result.
\begin{prop}\label{prop:variety_necessity}
    Let $I\subset\poly$ be a homogeneous ideal for which $[I]^\bot$ is essentially normal and of infinite dimension, then the zero variety of $I$ satisfies Condition A.
\end{prop}

\section{Construction of essentially normal quotient modules}
In this section, we establish a strategy for constructing essentially normal quotient modules, beginning with those associated to essentially quasi-prime ideals. Suppose $u\in\partial\mathbb D^d$, and $I_u$ is a homogenous ideal with $Z_{\mathbb{D}^d}(I_u)=V_u$. By \cite[Theorem 2.17]{WZ}, we know that $[I_u]^\perp$ is essentially normal provided $u\in\mathbb T^d$. To study essential normality of non-distinguished homogenous quotient modules, we have to investigate more on $[I_u]^\perp$.

 To continue, we need some terminologies. Let $H_1,H_2$ be Hilbert spaces. $A\in B(H_1,H_2)$ is called essentially bounded from below if there is a constant $c>0$ and compact operator $K$ such that $A^*A+K\geq cI_{H_1}$. Similarly $A\in B(H_1,H_2)$ is called an essential isometry if $I_{H_1}-A^*A$ is compact. $A$ is said essentially unitary if both $I_{H_2}-A A^*$ and $I_{H_1}-A^*A$ are compact. According to \cite{BDF}, every essentially unitary operator $A\in B(H)$ can be decomposed as $A=S+K$ where $K$ is compact, and $S$ is either a unitary operator, or a shift of multiplicity $n$, or the adjoint of a shift of multiplicity $n$, according to its Fredholm index being $0,n,-n$.

The following lemma can help us to prove the Fredholmness of some essential isometries.
\begin{lem}\label{lem:closedrange}
    If $A\in B(H_1,H_2)$ is an essential isometry, then the range of $A$ is closed.
\end{lem}
\begin{proof}
    Since $I_{H_1}-A^*A$ is compact, $A^*A$ is Fredholm and therefore has closed range. Then by the Inverse Mapping Theorem $A^*A$ is invertible on $(\ker A)^\bot$, and therefore bounded from below on $(\ker A)^\bot$. Consequently $A$ is bounded from below on $(\ker A)^\bot$ and has closed range.
\end{proof}
We begin our construction of essentially normal quotient modules with a careful consideration on the result of \cite{WZ}. Recall that the radical of an ideal $I\subset\poly$ is the ideal
$$\sqrt{I}=\{f\in\poly:\exists n\in\mathbb{N},f^n\in I\}.$$
\begin{lem}\label{lem:distinguished}
    Let $u\in\mathbb{T}^d$, and $I\subset\poly$ be an homogeneous ideal satisfying $Z_{\mathbb D^d}(I)=V_u$. Then $\mathcal{N}=[I]^\bot$ is essentially normal, and there is an essentially unitary operator $S\in B(\mathcal{N})$ and compact operators $K_i$ such that $S_{z_i}=u_iS+K_i$ for $1\leq i\leq d$. Consequently for $h(z)=d^{-1}\sum_{i=1}^d\bar{u}_iz_i$, the operator $S_h^*$ is essentially unitary on $\mathcal{N}$.
\end{lem}
\begin{proof}
     Recall that $J_u\subset\poly$ is  the ideal generated by $\{u_jz_i-u_iz_j:1\leq i,j\leq d\}$. Since $u\in\mathbb T_d$, $J_u=\sqrt{I}$. Since $\poly$ is a Neotherian ring, $J_u^n\subset I$ for some positive integer $n$. By \cite[Remark 2.8]{WZ}, $M_{z_1}^*$ is essentially isometric on $[J_u^n]^\bot$ and therefore $S_{z_1}^*$ is essentially isometric on $\mathcal{N}$. Then by \cite[Theorem 2.14]{WZ},
     $$P_\mathcal{N}-S_{z_1}^*S_{z_1}=P_\mathcal{N}-S_{z_1}S_{z_1}^*+[S_{z_1},S_{z_1}^*]$$
     is compact, and $S_{z_1}$ is essentially unitary. \cite[Remark 2.8]{WZ} also shows the compactness of $(u_iM_{z_i}^*-u_1M_{z_1}^*)\mid_{[J_u^n]^\bot}$ for $1\leq i\leq d$, which implies the compactness of $u_iS_{z_i}^*-u_1S_{z_1}^*$. Set $S=\bar{u}_1S_{z_1}$ then $K_i=S_{z_i}-u_i\bar{u}_1S_{z_1}$ is compact, and $S_{z_i}=u_iS+K_i$. Moreover,
     $$S_h^*=d^{-1}\sum_{i=1}^du_iS_{z_i}^*=S^*+d^{-1}\sum_{i=1}^dK_i^*,$$
     is essentially unitary.
\end{proof}

The following lemma reveals the structure of quotient modules associated to primary ideals with variety $V_u$. Recall that for $u\in\partial\mathbb{D}^d\backslash\mathbb{T}^d$, $\Lambda=\{i:u_i\in\mathbb{T},1\leq i\leq d\}$. For abbreviation, denote by $u_\Lambda=(u_{i_1},\cdots,u_{i_k})$ and $\mathbb Du_\Lambda=\{\lambda u_\Lambda,\lambda\in \mathbb D\}$. Moreover, the polynomial ring on the variable $z_\Lambda$ is denoted by $\mathbb C[z_\Lambda]$, and the Hardy module on $z_\Lambda$ is denoted by $H^2(\mathbb{D}^\Lambda)$.

\begin{lem}\label{lem:basic_type}
    Let $u\in\partial\mathbb{D}^d\backslash\mathbb{T}^d$, and $I\subset\poly$ be a homogeneous ideal with variety $V_u$. Suppose $k=|\Lambda|$, and set $I_0=I\cap \mathbb C[z_\Lambda]$. Then $I_0$ is a homogeneous ideal of $\mathbb{C}[z_\Lambda]$ with variety $Z_{\mathbb D^k}(I_0)=\mathbb{D}u_\Lambda$ and
    \begin{equation}\label{eq:series}
        [I]^\bot\subset\Big\{f\in\Hardy:f(z)=\sum_{\alpha\in\mathbb{Z}_+^{d-k}}z_{\Lambda^c}^\alpha f_\alpha(z_\Lambda),f_\alpha\in\mathcal{N}_0\Big\},
    \end{equation}
    where $\mathcal{N}_0=[I_0]^\bot$ is the quotient module of $H^2(\mathbb{D}^\Lambda)$ associated to $I_0$. Moreover if we denote by $$h(z)=k^{-1}\sum\limits_{i\in\Lambda}\bar{u}_iz_i,$$ then $S_h^*$ is essentially unitary and for $i,j\in\Lambda$, $u_iS_{z_i}^*-u_jS_{z_j}^*$ is compact.
\end{lem}
\begin{proof}
    Clearly $u_\Lambda\in Z(I_0)$, which implies $\mathbb{D}u_\Lambda\subset Z_{\mathbb{D}^k}(I_0)$. Conversely suppose $v\in Z_{\mathbb{D}^k}(I_0)$, then choose  $\tilde{v}\in\mathbb{D}^d$ such that $\tilde{v}_\Lambda=v$ and $\tilde{v}_j=u_jh(v),\forall j\in\Lambda^c$. Denote by $J\subset\poly$ the ideal generated by $\{z_j-u_jh:j\in\Lambda^c\}$,  then by the Hilbert Nullstellensatz there is an integer $N>0$ such that $J^N\subset I$. By a division algorithm, each $g\in I$ can be decomposed as $$g=\sum\limits_{j\in\Lambda^c}(z_j-u_j h)\tilde{g}_j+g_1,g_1\in\mathbb C[z_{\Lambda}],\tilde{g}_j\in\poly.$$
    Set $g_2=-\sum\limits_{j\in\Lambda^c}(z_j-u_j h) \tilde{g}_j\in J$, then we have
    $$g_1^N-g_2^N=g\sum_{k=0}^{N-1}g_1^kg_2^{N-1-k}\in I.$$
    This equality together with $g_2^N\in I$ implies $g_1^N\in I_0$. It follows from $g_1(v)^N=0$ and $g_2(\tilde{v})=0$ that $g(\tilde{v})=0$, which ensures $\tilde{v}\in Z_{\mathbb{D}^d}(I)=\mathbb{D}u$ and consequently $v\in\mathbb{D}u_\Lambda$. Hence we have proved $Z_{\mathbb D^k}(I_0)=\mathbb{D}u_\Lambda$.

    Notice that each function $f\in H^2(\mathbb D^d)$ has an expansion
    $$f(z)=\sum_{\alpha\in\mathbb{Z}_+^{d-k}}z_{\Lambda^c}^\alpha f_\alpha(z_\Lambda),f_\alpha\in H^2(\mathbb{D}^\Lambda).$$
    For  $f\in [I]^\bot$, since $M_g^*f=0$ for each $g\in I_0$ we have $f_\alpha\in\mathcal{N}_0,\forall\alpha\in\mathbb{Z}_+^{d-k}$. Hence (\ref{eq:series}) is proved.

     For $j\in\Lambda^c$, since $h_j^N(z)=(z_j-u_jh(z))^N\in I$ we have for $n\geq N$ that
    \begin{eqnarray*}
        S_{z_j}^{*n}&=&[\bar{u}_jS_h^*+(S_{z_j}-\bar{u}_jS_h^*)]^n\\
        &=&\sum_{l=0}^n\binom{n}{l}\bar{u}_j^{n-l}S_h^{*n-l}(S_{z_j}-\bar{u}_jS_h^*)^l\\
        &=&\sum_{l=0}^{N-1}\binom{n}{l}\bar{u}_j^{n-l}S_h^{*n-l}(S_{z_j}-\bar{u}_jS_h^*)^l
    \end{eqnarray*}
    Then for $n\geq2N$ it holds
    \begin{eqnarray}\label{eq:basictype1}
        ||S_{z_j}^n||&\leq&\sum_{l=0}^{N-1}\binom{n}{l}|\bar{u}_j|^{n-l}(1+|u_j|)^l\notag\\
        &\leq&N\binom{n}{N-1}|\bar{u}_j|^{n-N+1}(1+|u_j|)^{N-1}.
    \end{eqnarray}

    For positive integers $n$, denote by $Q_n$ the projection from $[I]^\bot$ to the subspace of $\Hardy$ spanned by $\{z_{\Lambda^c}^\alpha g:\alpha\in\mathbb{Z}_+^{d-k},\max_i|\alpha_i|<n,g\in\mathcal{N}_0\}$. Then by (\ref{eq:basictype1}) it holds for $f\in[I]^\bot$ that
    \begin{eqnarray*}
        ||f-Q_nf||^2&=&||\sum_{\max_j\alpha_j>n}z_{\Lambda^c}^\alpha f_\alpha||^2\\
        &\leq&\sum_{j=1}^{d-k}||\sum_{\alpha_j>n}z_{\Lambda^c}^\alpha f_\alpha||^2\\
        &=&\sum_{j\in\Lambda^c}||S_{z_j}^{*n}f||^2\\
        &\leq&||f||^2 N^2\binom{n}{N-1}^2\sum_{j\in\Lambda^c}|\bar{u}_j|^{2n-2N+2},
    \end{eqnarray*}
    which converges uniformly for $||f||\leq1$ as $n\to\infty$. Then we conclude $\lim_{n\to\infty}Q_n=P_{[I]^\bot}$. By Lemma \ref{lem:distinguished} $M_h^*\mid_{\mathcal{N}_0}$ is essentially isometric, and $(u_iM_{z_i}^*-u_jM_{z_j}^*)\mid_{\mathcal{N}_0}$ is compact for $i,j\in\Lambda$. Then for any integer $n>0$ and sequence $\{g_m\}$ of unit vectors in $[I]^\bot$ that converges weakly to $0$, we obtain by Lemma \ref{lem:distinguished} that for $i\in\Lambda$
    $$\liminf_{m\to\infty}||S_h^*g_m||\geq\liminf_{m\to\infty}||Q_nM_h^*g_m||\geq\liminf_{m\to\infty}||M_h^*Q_ng_m||=\liminf_{m\to\infty}||Q_ng_m||,$$
    which implies $\lim_{m\to\infty}||S_h^*g_m||=1$, proving that $S_{z_i}^*$ is essentially isometric. Similarly we have
    $$\lim_{m\to\infty}||(u_iM_{z_i}^*-u_jM_{z_j}^*)g_m||=\lim_{n\to\infty}\lim_{m\to\infty}||Q_n(u_iM_{z_i}^*-u_jM_{z_j}^*)g_m||=0,$$
    which induces the compactness of $u_iS_{z_i}^*-u_jS_{z_j}^*$.

    Since $S_h^*$ is essentially isometric, $P_\mathcal{N}-S_hS_h^*$ is compact and therefore $\ker S_h^*$ must be of finite dimension. For homogeneous $f\in\coker S_h^*$, it holds for each $g\in[I]^\bot$ that
    $$\langle hf,g\rangle=\langle f,S_h^*g\rangle=0,$$
    which induces $hf\in I$. Then since $I$ is primary and $h\notin\sqrt{I}$, $f$ belongs to $I$. Therefore
    $$\ker S_h=\coker S_h^*=\{0\},$$
    and $S_h^*$ is Fredholm by Lemma \ref{lem:closedrange}. Then $S_h$ is the inverse of $S_h^*$ in the Calkin algebra on $\mathcal{N}$, which implies the compactness of $P_{\mathcal{N}}-S_h^*S_h$, and ensures that $S_h$ is essentially unitary.
\end{proof}
The following proposition ensures the essential normality of essentially quasi-prime quotient modules.
\begin{prop}\label{prop:basictype}
    Let $u\in\partial\mathbb{D}^d$, and $I\subset\poly$ be an essentially quasi-prime homogeneous ideal with variety $V_u$.  Then $\mathcal{N}=[I]^\bot$ is essentially normal. Moreover there is an essentially unitary operator $S\in B(\mathcal{N})$ and compact operators $K_i$ such that
    \begin{equation}\label{eq:basictype2}
        S_{z_i}=u_iS+K_i,\forall 1\leq i\leq d.
    \end{equation}
\end{prop}
\begin{proof}
    If $u\in\mathbb{T}^d$, then the conclusion is contained in Lemma \ref{lem:distinguished}. So we assume $u\in\partial\mathbb{D}^d\backslash\mathbb{T}^d$ in the remaining of the proof. By Lemma \ref{lem:basic_type} $S_h^*$ is essentially unitary, and if $i,j\in\Lambda$ then $\bar{u}_iS_{z_i}-\bar{u}_jS_{z_j}$ is compact. It follows that
    $$[S_h^*,S_h]=(S_h^*S_h-P_{\mathcal{N}})-(S_hS_h^*-P_{\mathcal{N}})\in\mathcal{K}.$$
    This equality together with the compactness of $u_iS_{z_i}^*-u_jS_{z_j}^*$ gives
    $$[S_{z_j}^*,S_{z_i}]\in\mathcal{K},\forall i,j\in\Lambda.$$

    Set $S=S_h$ and $K_i=S_{z_i}-u_iS_h$ for $i\in\Lambda$, then $K_i$ is compact and $S_{z_i}=u_iS+K_i$. For $j\in\Lambda^c$, clearly $h_j\in\sqrt{I}$ and there $h_jh^n\in\sqrt{I}$ for every natural number $n$. Since $\dim\sqrt{I}/I'<+\infty$, there is an integer $N>0$ such that $f\in I'$ provided that $f\in\sqrt{I}$ is homogeneous and $\deg f\geq N$. Therefore $h_jh^N\in I'$, and there exists $f_j\in J_u$ such that $h_jh^N-f_j\in I$. As a consequence
    $S_{h_j}^*S_h^{*N}=S_{f_j}^*$, which is compact by Lemma \ref{lem:basic_type}. Then since $S_h^*$ is essentially unitary $S_{h_j}^*$ must be compact, completing the proof of (\ref{eq:basictype2}). As a consequence, $[I]^\bot$ is essentially normal.
\end{proof}
\begin{rem}\label{rem:Fredholmindex}
    In either Lemma \ref{lem:distinguished} or Proposition \ref{prop:basictype}, $S_h$ is essentially unitary, and therefore Fredholm. Clearly $1\in\ker S_h^*$ and $\dim\ker S_h^*>0$. Take an arbitrary homogeneous $g\in\ker S_h$, then $hg\in I$. But since $h(u)\neq0$, $h$ cannot belong to $\sqrt{I}$. Then $g\in I$ and hence $\ker S_h=\{0\}$. Therefore the Fredholm index of $S_h$ is nonzero.
\end{rem}
\begin{cor}\label{cor:basictype}
    Let $\mathcal{N}$ be as in Proposition \ref{prop:basictype}. Then $P_\mathcal{N}^\bot M_{z_i}P_\mathcal{N}$ is compact for $i\in\Lambda$.
\end{cor}
\begin{proof}
    By Proposition \ref{prop:basictype}, $S_{z_i}$ is essentially unitary for $i\in\Lambda$, and then $P_\mathcal{N}-S_{z_i}^*S_{z_i}$ is compact. Therefore since
    $$0\leq P_\mathcal{N}M_{z_i}^*P_\mathcal{N}^\bot M_{z_i}P_\mathcal{N}=P_\mathcal{N}M_{z_i}^*M_{z_i}P_\mathcal{N}-S_{z_i}^*S_{z_i}\leq P_\mathcal{N}-S_{z_i}^*S_{z_i},$$
    $P_\mathcal{N}M_{z_i}^*P_\mathcal{N}^\bot M_{z_i}P_\mathcal{N}$ is compact.
\end{proof}
To prove the essential normality of sums of essentially quasi-prime quotient modules, we need the concept of asymptotical orthogonality. As in \cite{GWp1}, if subspaces $N_1,N_2$ of the Hilbert space $H$ satisfy that $P_{N_1}P_{N_2}$ is compact, then $N_1$ and $N_2$ are said to be asymptotically orthogonal to each other. As usual, let $\pi: B(H)\to B(H)/K(H)$ be the quotient map.
\begin{lem}
    Let $u,v$ be two different points in $\partial\mathbb{D}^d$, and $1\leq i,j\leq d$ satisfy $u_i,v_i,u_j,v_j\in\mathbb{T}$ and $\bar{u}_i v_i\neq\bar{u}_j v_j$. Let $I_u,I_v\subset\poly$ be essentially quasi-prime homogeneous ideals with variety $V_u,V_v$ respectively, then $[I_u]^\bot$ is \ao to $[I_v]^\bot$.
\end{lem}
\begin{proof}
    By Corollary \ref{cor:basictype}, both $P_{[I_u]^\bot}M_{z_i}^*P_{[I_u]}$ and $P_{[I_v]}M_{z_i}P_{[I_v]^\bot}$ are compact. Then by Proposition \ref{prop:basictype} we obtain
    \begin{eqnarray}\label{eq:ao}
        \pi(P_{[I_u]^\bot}P_{[I_v]^\bot})&=&\pi(P_{[I_u]^\bot}P_{[I_v]^\bot}M_{z_i}^*P_{[I_v]^\bot}M_{z_i}P_{[I_v]^\bot})\notag\\
        &=&\pi(P_{[I_u]^\bot}M_{z_i}^*P_{[I_v]^\bot}M_{z_i}P_{[I_v]^\bot})\notag\\
        &=&\pi(P_{[I_u]^\bot}M_{z_i}^*P_{[I_u]^\bot}P_{[I_v]^\bot}M_{z_i}P_{[I_v]^\bot})\notag\\
        &=&\pi(\bar{u}_i v_iS_u^*S_v),
    \end{eqnarray}
    and similarly $\pi(P_{[I_u]^\bot}P_{[I_v]^\bot})=\pi(\bar{u}_j v_jS_u^*S_v)$. Hence $(\bar{u}_i v_i-\bar{u}_j v_j)\pi(S_u^*S_v)=0$, inducing that $S_u^*S_v$ is compact. Then the conclusion of the lemma follows from (\ref{eq:ao}).
\end{proof}
\begin{lem}
    Let $u,v$ be two different points in $\partial\mathbb{D}^d$, and $1\leq i\leq d$ satisfy $u_i\in\mathbb{T}$ and $v_i\in\mathbb{D}$. Let $I_u,I_v\subset\poly$ be essentially quasi-prime homogeneous ideals with variety $V_u,V_v$ respectively, then $[I_u]^\bot$ is \ao to $[I_v]^\bot$.
\end{lem}
\begin{proof}
    By Corollary \ref{cor:basictype} $P_{[I_u]}M_{z_i}P_{[I_u]^\bot}$ is compact. If $P_{[I_u]}M_{z_i}^nP_{[I_u]^\bot}$ is compact, then $$P_{[I_u]}M_{z_i}^{n+1}P_{[I_u]^\bot}=P_{[I_u]}M_{z_i}P_{[I_u]^\bot}M_{z_i}^nP_{[I_u]^\bot}+P_{[I_u]}M_{z_i}P_{[I_u]}M_{z_i}^nP_{[I_u]^\bot}\in\mathcal{K}.$$
    Hence by inductive method  $P_{[I_u]}M_{z_i}^nP_{[I_u]^\bot}$ is compact for any positive integer $n$. Then by Lemma \ref{lem:basic_type} and Proposition \ref{prop:basictype} it holds for any integer $n>0$ that
    \begin{eqnarray}
        \pi(P_{[I_v]^\bot}P_{[I_u]^\bot})&=&\pi(P_{[I_v]^\bot}P_{[I_u]^\bot}M_{z_i}^nP_{[I_u]^\bot}M_{z_i}^{*n}P_{[I_u]^\bot})\\
        &=&\pi(P_{[I_v]^\bot}M_{z_i}^nM_{z_i}^{*n}P_{[I_u]^\bot})\notag\\
        &=&\bar{u}_i^nv_i^n\pi(S_v^nS_u^{n*}).\notag
    \end{eqnarray}
    Then $||\pi(P_{[I_v]^\bot}P_{[I_u]^\bot})||\leq|v_i|^n$ for every positive integer $n$, inducing that $\pi(P_{[I_v]^\bot}P_{[I_u]^\bot})=0$.
\end{proof}
Combining these two lemmas, we immediately get the following result.
\begin{cor}\label{cor:ao}
    Let $u,v\in\partial\mathbb{D}^d$, such that $V_u\neq V_v$ and $V_u\cup V_v$ satisfies Condition A. If $I_u,I_v\subset\poly$ are essentially quasi-prime homogeneous ideals with variety $V_u,V_v$ respectively, then $[I_u]^\bot$ is \ao to $[I_v]^\bot$.
\end{cor}
Asymptotically orthogonal quotient modules posses the following additive property.
\begin{lem}\label{lem:asum}
    If essentially normal quotient modules $\mathcal{N}_1,\mathcal{N}_2,\mathcal{N}_3$ of $\Hardy$ are asymptotically orthogonal to each other, then the quotient module $\mathcal{N}=\mathcal{N}_1+\mathcal{N}_2$ is essentially normal and asymptotically orthogonal to $\mathcal{N}_3$.
\end{lem}
\begin{proof}
    By \cite[Proposition 3.1]{GWp1}, $\mathcal{N}$ is closed and therefore a quotient module. By \cite[Theorem 3.3]{GWp1}, if $1\leq j,k\leq d$ then $[S_{z_j}^*+\bar{i}^lS_{z_k}^*,S_{z_j}+i^lS_{z_k}]$ is compact for $l=0,1,2,3$. Then by the identity
    $$[S_{z_j}^*,S_{z_k}]=\sum_{l=0}^3i^l[S_{z_j}^*+\bar{i}^lS_{z_k}^*,S_{z_j}+i^lS_{z_k}]$$
    we obtain the compactness of $[S_{z_j}^*,S_{z_k}]$, which leads to the essential normality of $\mathcal{N}$. By the definition of asymptotical orthogonality, $P_{\mathcal{N}_3}P_{\mathcal{N}_1}$ and $P_{\mathcal{N}_3}P_{\mathcal{N}_2}$ are compact. By \cite[Proposition 3.2]{GWp1}, $P_\mathcal{N}-P_{\mathcal{N}_1}-P_{\mathcal{N}_2}$ is compact, and therefore
    $$P_{\mathcal{N}_3}P_\mathcal{N}=P_{\mathcal{N}_3}(P_\mathcal{N}-P_{\mathcal{N}_1}-P_{\mathcal{N}_2})+P_{\mathcal{N}_3}P_{\mathcal{N}_1}+P_{\mathcal{N}_3}P_{\mathcal{N}_2}\in\mathcal{K},$$
    indicating the asymptotical orthogonality between $\mathcal{N}$ and $\mathcal{N}_3$.
\end{proof}
\begin{thm}\label{thm:construction}
    Let $I\subset\poly$ be a homogeneous ideal with primary decomposition $\bigcap_{j=1}^m I_{u_j}$ with $Z_{\mathbb{D}^d}(I_{u_j})=V_{u_j}$ where $u_j\in\partial\mathbb{D}^d$, and assume that $Z_{\mathbb{D}^d}(I)$ satisfies Condition A. If each $I_{u_j}$ is essentially quasi-prime, then $[I]^\bot$ is essentially normal.
\end{thm}
\begin{proof}
    By Proposition \ref{prop:basictype} each $[I_{u_j}]^\bot$ is essentially normal, and by Corollary \ref{cor:ao} $[I_{u_i}]^\bot$ is \ao to $[I_{u_j}]^\bot$ whenever $i\neq j$. Then the essential normality of $[I]^\bot$ follows from Lemma \ref{lem:asum}.
\end{proof}

\section{The ideal associated to an essentially normal quotient module}
The present section is devoted to prove that, Theorem \ref{thm:construction} produces all the essentially normal homogeneous quotient modules of $\Hardy$.

Throughout this section we assume that the ideal $I\subset\poly$ is homogeneous, and denote by $\mathcal{N}=[I]^\bot$. By proposition \ref{prop:variety_necessity}, if $\mathcal{N}$ is essentially normal then $Z(I)$ satisfies Condition A. Write its  primary decomposition as
$$I=\bigcap_{k=1}^m I_k,$$
with the zero variety $Z_{\mathbb D^d}(I_k)=V_{u_k}$ for some $u_k\in\partial\mathbb{D}^d$. Denote by $\mathcal{N}_k=[I_k]^\bot$ for $1\leq k\leq m$, then clearly $\mathcal{N}_k\subset\mathcal{N}$.

A basic fact is that, even if $Z_{\mathbb D^d}(I)$ satisfies Condition A, $[I]^\bot$ might not be essentially normal.
\begin{exam}
    Consider the quotient module of $H^2(\mathbb{D}^2)$ associated to the ideal $I=z_1^2\mathbb{C}[z_1,z_2]$. Clearly $Z_{\mathbb D^2}(I)=V_{(0,1)}$, for which Condition A holds. Then
    $$[\sqrt{I}]^\bot=\overline{\span}\{z_2^n:n=0,1,\ldots\},$$
    being essentially normal. It is routine to verify that $[S_{z_1},S_{z_1}^*]$ is the orthogonal projection onto $[I]^\bot\ominus[\sqrt{I}]^\bot$, which cannot be compact. Notice that $[I]^\bot\ominus[\sqrt{I}]^\bot$ is of infinite dimension.
\end{exam}
We have the following necessary conditions to essential normality of quotient modules.
\begin{lem}
    If $\mathcal{N}$ is essentially normal, then $S_{f}$ is compact whenever
    $f\in\sqrt{I}$.
\end{lem}
\begin{proof}
    By definition, there is a natural number $n$ such that $f^n\in I$, and therefore $S_f^{*n}S_f^n=0$. Then since $[S_f^*,S_f]$ is compact, $(S_f^*S_f)^n=(S_f^*S_f)^n-S_f^{*n}S_f^n$ also be compact, leading to the compactness of $S_f$.
\end{proof}
\begin{lem}\label{lem:compactness}
    Suppose $\mathcal{N}$ is essentially normal, and $1\leq k\leq m$. Then $S_f^*P_{\mathcal{N}_k}$ is compact whenever $f\in\sqrt{I_k}$.
\end{lem}
\begin{proof}
    Since $f\in\sqrt{I_k}$ there is an integer $N>0$ such that $f^{2^N}\in I_k$. Therefore
    $$P_{\mathcal{N}_k}S_f^{2^N}S_f^{*2^N}P_{\mathcal{N}_k}=0\in\mathcal{K}.$$
    The essential normality of $\mathcal{N}$ assures that $[S_f^*,S_f]$ is compact, and then we have
    $$P_{\mathcal{N}_k}(S_fS_f^*)^{2^{N-1}}P_{\mathcal{N}}(S_fS_f^*)^{2^{N-1}}P_{\mathcal{N}_k}=P_{\mathcal{N}_k}(S_fS_f^*)^{2^N}P_{\mathcal{N}_k}\in\mathcal{K},$$
    and therefore $P_{\mathcal{N}}(S_fS_f^*)^{2^{N-1}}P_{\mathcal{N}_k}$ is compact and so $P_{\mathcal{N}_k}(S_fS_f^*)^{2^{N-1}}P_{\mathcal{N}_k}$ is.

    By similar procedure we can recursively obtain the compactness of $P_{\mathcal{N}_k}(S_fS_f^*)^{2^{n}}P_{\mathcal{N}_k},n=N-2,\ldots,0$. Then the lemma follows from $P_{\mathcal{N}_k}S_fS_f^*P_{\mathcal{N}_k}\in\mathcal{K}$.
\end{proof}
The following lemma is a consequence of Lemma \ref{lem:basic_type}, and we admit the notations from its proof.
\begin{lem}\label{lem:prime}
    Let $u\in\partial\mathbb{D}^d\backslash\mathbb{T}^d$. Fix $j_0\in\Lambda^c$, and denote by $\Lambda_1=\{1,\ldots,d\}\backslash\{j_0\}$. Suppose $I\subset\poly$ is the prime ideal with variety $V_u$. Set $I_0=I\cap\mathbb{C}[z_{\Lambda_1}]$, then $I_0$ is the prime ideal of $\mathbb{C}[z_{\Lambda_1}]$ with variety $Z_{\mathbb{D}^{d-1}}(I_0)=\mathbb{D}u_{\Lambda_1}$ such that
    $$[I]^\bot=\big\{\sum_{n=0}^\infty z_{j_0}^n\bar{u}_{j_0}^n M_h^{*n}f_0(z_{\Lambda_1}):f_0\in\mathcal{N}_0\big\},$$
    where $\mathcal{N}_0=[I_0]^\bot$ is the quotient module of $H^2(\mathbb{D}^{\Lambda_1})$ associated to $I_0$. Moreover $S_h$ is essentially unitary on $[I]^\bot$.
\end{lem}
\begin{proof}
    If $f_0\in\mathcal{N}_0$, set
    $$f(z)=\sum_{n=0}^\infty z_{j_0}^n\bar{u}_{j_0}^n M_h^{*n}f_0(z_{\Lambda_1}).$$
    Since $|\bar{u}_{j_0}|<1$ we have $f\in\Hardy$. If $g\in I_0$, then since $M_g^*f_0=0$ we have $M_g^*f=0$. It is routine to verify that $M_{h_j}^*f=0$ for $h_j=z_j-u_jh,\forall j\in\Lambda^c$. Therefore $M_g^*f=0,\forall g\in I$ since $I$ is generated by $I_0$ and $\{h_j:j\in\Lambda^c\}$. Hence $f\in[I]^\bot$.

    Conversely, each $f\in[I]^\bot$ can be represented as
    $$f(z)=\sum_{n=0}^\infty z_{j_0}^nf_n(z),f_n\in H^2(\mathbb{D}^{\Lambda_1}).$$
    Since $(M_{z_{j_0}}^*-\bar{u}_{j_0}M_h^*)f=0$, we have
    $$\sum_{n=1}^\infty z_{j_0}^{n-1}f_n(z_{\Lambda_1})=\sum_{n=0}^\infty z_{j_0}^n\bar{u}_{j_0}M_h^*f_n(z_{\Lambda_1})$$
    and therefore $f_{n+1}=\bar{u}_{j_0}M_h^*f_n,\forall n\in\mathbb{Z}_+$. Hence $f_n=\bar{u}_{j_0}^nM_h^{*n}f_0$ and
    \begin{equation}\label{eq:prime}
        f(z)=\sum_{n=0}^\infty z_{j_0}^n\bar{u}_{j_0}^nM_h^{*n}f_0(z_{\Lambda_1}).
    \end{equation}
    For $g\in I_0$, we have $M_g^*f=0$ and consequently $M_g^*f_0=0$ by (\ref{eq:prime}), which ensures $f_0\in\mathcal{N}_0$.

    Since $I$ is prime, it follows directly that $I_0$ is prime. By Lemma \ref{lem:basic_type} $Z_{\mathbb{D}^{d-1}}(I_0)=\mathbb{D}u_{\Lambda_1}$, and by Proposition \ref{prop:basictype} $S_h$ is essentially unitary. The proof is completed.
\end{proof}
\begin{lem}\label{lem:bddbelow1}
    Let $u\in\partial\mathbb{D}^d\backslash\mathbb{T}^d$. Fix $j_0\in\Lambda^c$, and denote by $\Lambda_1=\{1,\ldots,d\}\backslash\{j_0\}$. Suppose $I\subset\poly$ is an ideal with variety $V_u$ such that $J_u\subset I$, and $h_j\in I,\forall j\in\Lambda^c\backslash\{j_0\}$, and $h_{j_0}^2\in I$. Set $I_0=\sqrt{I}\cap\mathbb{C}[z_{\Lambda_1}]$, then $I_0$ is the prime ideal of $\mathbb{C}[z_{\Lambda_1}]$ with variety $Z_{\mathbb{D}^{d-1}}(I_0)=\mathbb{D}u_{\Lambda_1}$ such that
    \begin{eqnarray}\label{eq:subset}[I]^\bot\subset\big\{\sum_{n=0}^\infty z_{j_0}^n\bar{u}_{j_0}^n M_h^{*n}f_0(z)+\sum_{n=1}^\infty nz_{j_0}^n\bar{u}_{j_0}^{n-1} M_h^{*(n-1)}g_0(z):f_0,g_0\in\mathcal{N}_0\big\},\end{eqnarray}
    where $\mathcal{N}_0=[I_0]^\bot$ is the quotient module of $H^2(\mathbb{D}^{\Lambda_1})$ associated to $I_0$.  Moreover $S_h$ is essentially unitary on $[I]^\bot$.
\end{lem}
\begin{proof}
    By Lemma \ref{lem:prime}, $I_0$ is the prime ideal of $\mathbb{C}[z_{\Lambda_1}]$ associated to the variety $\mathbb{D}u_{\Lambda_1}$, and
    \begin{equation}\label{eq:bddbelow1_1}
        [\sqrt{I}]^\bot=\Big\{f\in\Hardy:f=\sum_{n=0}^\infty z_{j_0}^n\bar{u}_{j_0}^n M_h^{*n}f_0,f_0\in\mathcal{N}_0\Big\}.
    \end{equation}
    If $f\in[I]^\bot$, then $M_{h_{j_0}}^{*2}f=0$, namely
    \begin{equation}\label{eq:bddbelow1_2}
        M_{z_{j_0}}^{*2}f=2\bar{u}_dM_{z_{j_0}}^*M_h^*f-\bar{u}_d^2M_h^{*2}f.
    \end{equation}
    Applying (\ref{eq:bddbelow1_2}) to the expansion
    $$f(z)=\sum_{n=0}^\infty z_{j_0}^n f_n(z),f_n\in H^2(\mathbb{D}^{\Lambda_1}),$$
    and comparing the terms involving $z_{j_0}^n$, we find
    $$f_{n+2}=2\bar{u}_{j_0}M_h^*f_{n+1}-\bar{u}_{j_0}^2M_h^{*2}f_n,\forall n\in\mathbb{Z}_+.$$
    Thus $f$ is determined by $f_0$ and $f_1$, and we can solve that
    \begin{equation}\label{eq:bddbelow1_3}
        f=\sum_{n=0}^\infty z_{j_0}^n\bar{u}_{j_0}^n M_h^{*n}f_0+\sum_{n=1}^\infty nz_{j_0}^n\bar{u}_{j_0}^{n-1} M_h^{*(n-1)}g_0,
    \end{equation}
    where $g_0:=f_1-\bar{u}_{j_0}M_h^*f_0$. Clearly $I_0\subset I$, and therefore for any $g\in I_0$ we have $M_g^*f=0$. Then it follows from (\ref{eq:bddbelow1_3}) that
    $M_g^*f_0=0$, namely $f_0\in\mathcal{N}_0$. Then since
    $$M_{h_{j_0}}^*f=\sum_{n=0}^\infty z_{j_0}^n\bar{u}_{j_0}^n M_h^{*n}g_0\in[I]^\bot,$$
    we have $g\in\mathcal{N}_0$.

    By Lemma \ref{lem:basic_type} $S_h^*$ is essentially unitary. The proof of the lemma is completed.
\end{proof}
The following lemma plays the key role in finding the necessary condition for the essential normality.
\begin{lem}\label{lem:bddbelow2}
     Suppose the ideal $I\subset\poly$ satisfies the condition of Lemma \ref{lem:bddbelow1}. Then $M_{h_{j_0}}^*$ is essentially bounded from below on $[I]^\bot\ominus[\sqrt{I}]^\bot$.
\end{lem}
\begin{proof}
    If $[I]^\bot\ominus[\sqrt{I}]^\bot$ is of finite dimension then the conclusion is trivial, and therefore we assume $\dim[I]^\bot\ominus[\sqrt{I}]^\bot=\infty$ in the rest of the proof. Since $h\notin I_0$, we have $$\ker P_{\mathcal{N}_0}M_h\mid_{\mathcal{N}_0}=\{0\}.$$ Then since $P_{\mathcal{N}_0}M_h\mid_{\mathcal{N}_0}$ is Fredholm, $\range M_h^*\mid_{\mathcal{N}_0}=\mathcal{N}_0$. By Lemma \ref{lem:bddbelow1} each $f\in[I]^\bot\ominus[\sqrt{I}]^\bot$ has the expansion
    $$f=\sum_{n=0}^\infty z_{j_0}^n\bar{u}_{j_0}^n M_h^{*n}f_0+\sum_{n=1}^\infty nz_{j_0}^n\bar{u}_{j_0}^{n-1} M_h^{*(n-1)}g_0\in[I]^\bot,f_0,g_0\in\mathcal{N}_0.$$
    There is a unique $f_*\in\mathcal{N}_0\ominus\ker(M_h^*\mid_{\mathcal{N}_0})$ such that $M_h^*f_*=g_0$. Define
    $$Af:=-\frac{u_{j_0}}{1-|u_{j_0}|^2}\sum_{n=0}^\infty z_{j_0}^n\bar{u}_{j_0}^n M_h^{*n}f_*+\sum_{n=1}^\infty nz_{j_0}^n\bar{u}_{j_0}^{n-1} M_h^{*(n-1)}g_0,$$
    then by Lemma \ref{lem:prime}, $f-Af\in[\sqrt{I}]^\bot$.

    Let $\{f^{(n)}\}$ and $\{h^{(n)}\}$ be sequences of unit vectors in $[I]^\bot\ominus[\sqrt{I}]^\bot$ and $[\sqrt{I}]^\bot$ respectively, that converge to $0$ in the weak topology. We find successively that $\{f_0^{(n)}\},\{g_0^{(n)}\},\{f_*^{(n)}\}$ converge to $0$ weakly. Since $\lim_{m\to\infty}Q_m=P_{[I]^\bot}$ we have
    \begin{eqnarray*}
        &&\lim_{n\to\infty}\langle Af^{(n)},h^{(n)}\rangle\\
        &=&\lim_{n\to\infty}\lim_{m\to\infty}\langle Af^{(n)},Q_mh^{(n)}\rangle\\
        &=&\lim_{m\to\infty}\lim_{n\to\infty}\sum_{l\leq m}\langle-\frac{u_{j_0}}{1-|u_{j_0}|^2}\bar{u}_{j_0}^n M_h^{*l}f_*^{(n)}+l\bar{u}_{j_0}^{l-1} M_h^{*(l-1)}g_0^{(n)},\bar{u}_{j_0}^l M_h^{*l}h_0^{(n)}\rangle\\
        &=&\lim_{m\to\infty}\lim_{n\to\infty}\sum_{l\leq m}(-\frac{u_{j_0}}{1-|u_{j_0}|^2}\langle\bar{u}_{j_0}^l f_*^{(n)},\bar{u}_{j_0}^l g_0^{(n)}\rangle+\langle l\bar{u}_{j_0}^{l-1}f_*^{(n)},\bar{u}_{j_0}^l h_0^{(n)}\rangle)\\
        &=&\lim_{m\to\infty}\lim_{n\to\infty}\sum_{l\leq m}(-\frac{u_{j_0}|\bar{u}_{j_0}|^{2l}}{1-|u_{j_0}|^2}+lu_{j_0}|\bar{u}_{j_0}^{l-1}|^2)\langle f_*^{(n)},h_0^{(n)}\rangle\\
        &=&0,
    \end{eqnarray*}
    where the last equality follows from
    $$\sum_{l=0}^\infty\frac{u_{j_0}|\bar{u}_{j_0}|^{2l}}{1-|u_{j_0}|^2}=\sum_{l=1}^\infty lu_{j_0}|\bar{u}_{j_0}|^{2l-2}=\frac{u_{j_0}}{1-|u_{j_0}|^2}\frac{1}{1-|u_{j_0}|^2}.$$
    As a consequence we have
    $$\lim_{n\to\infty}(f^{(n)}-Af^{(n)})=\lim_{n\to\infty}P_{[\sqrt{I}]^\bot}(f^{(n)}-Af^{(n)})=-\lim_{n\to\infty}P_{[\sqrt{I}]^\bot}Af^{(n)}=0.$$
    Therefore
    \begin{eqnarray*}
        &&\lim_{n\to\infty}||S_{h_{j_0}}^*f^{(n)}||^2\cdot||g_0^{(n)}||^{-2}\\
        &=&\lim_{n\to\infty}||S_{h_{j_0}}^*Af^{(n)}||^2\cdot||g_0^{(n)}||^{-2}\\
        &=&\lim_{n\to\infty}||\sum_{l=0}^\infty z_{j_0}^l\bar{u}_{j_0}^l M_h^{*l}g_0^{(n)}||^2\cdot||g_0^{(n)}||^{-2}\\
        &=&\sum_{l=0}^\infty|\bar{u}_{j_0}|^{2l}\\
        &=&\frac{1}{1-|u_{j_0}|^2},
    \end{eqnarray*}
    and
    \begin{eqnarray*}
        &&\lim_{n\to\infty}||f^{(n)}||^2\cdot||g_0^{(n)}||^{-2}\\
        &=&\lim_{n\to\infty}||Af^{(n)}||^2\cdot||f_*^{(n)}||^{-2}\\
        &=&\lim_{n\to\infty}||-\frac{u_{j_0}}{1-|u_{j_0}|^2}\sum_{l=0}^\infty z_{j_0}^l\bar{u}_{j_0}^l M_h^{*l}f_*^{(n)}+\sum_{l=1}^\infty lz_{j_0}^l\bar{u}_{j_0}^{l-1}M_h^{*l}f_*^{(n)}||^2\cdot||f_*^{(n)}||^{-2}\\
        &=&\sum_{l=1}^\infty|-\frac{|u_{j_0}|^2}{1-|u_{j_0}|^2}\bar{u}_{j_0}^{l-1}+l\bar{u}_{j_0}^{l-1}|^2+\frac{|u_{j_0}|^2}{(1-|u_{j_0}|^2)^2}\\
        &=&\sum_{l=0}^\infty |-\frac{|u_{j_0}|^2}{1-|u_{j_0}|^2}\bar{u}_{j_0}^l+(l+1)\bar{u}_{j_0}^l|^2+\frac{|u_{j_0}|^2}{(1-|u_{j_0}|^2)^2}\\
        &=&\sum_{l=0}^\infty \left[(l+1)(l+2)-(l+1)\frac{1+|u_{j_0}|^2}{1-|u_{j_0}|^2}+\frac{|u_{j_0}|^4}{(1-|u_{j_0}|^2)^2}\right]|u_{j_0}|^{2l}+\frac{|u_{j_0}|^2}{(1-|u_{j_0}|^2)^2}\\
        &=&\frac{2-(1+|u_{j_0}|^2)+|u_{j_0}^4|+|u_{j_0}|^2(1-|u_{j_0}|^2)}{(1-|u_{j_0}|^2)^3}\\
        &=&\frac{1}{(1-|u_{j_0}|^2)^3},
    \end{eqnarray*}
    inducing that
    $$\lim_{n\to\infty}\frac{||S_{h_{j_0}}^*f^{(n)}||}{||f^{(n)}||}=1-|u_{j_0}|^2.$$
    Thus $S_{h_{j_0}}^*$ is essentially bounded from below on $[I]^\bot\ominus[I']^\bot$.
\end{proof}
\begin{cor}\label{cor:bdb}
    Let $u\in\partial\mathbb{D}^d\backslash\mathbb{T}^d$. Define polynomials $h(z)=\sum_{i\in\Lambda}\bar{u}_iz_i$ and $h_j(z)=z_j-u_jh(z)$ for $j\in\Lambda^c$. Let $I\subset\poly$ be an ideal of variety $V_u$ such that $J_u\subset I$, and $h_ih_j\in I,\forall i,j\in\Lambda^c$. Then $\sum_{j\in\Lambda^c}S_{h_j}S_{h_j}^*$ is essentially bounded from below on $[I]^\bot\ominus[\sqrt{I}]^\bot$.
\end{cor}
\begin{proof}
    Without loss of generality we assume $\Lambda=\{1,\ldots,k\}$. The conclusion is trivial if $\dim[I]^\bot\ominus[\sqrt{I}]^\bot<\infty$, so we assume $\dim[I]^\bot\ominus[\sqrt{I}]^\bot=\infty$ in the remaining of the proof. For $j\in\Lambda^c$ denote by $I_j\subset\poly$ the ideal generated by $I$ and $h_{k+1},\ldots,\hat{h}_j,\ldots,h_d$, where $\hat{h}_j$ means $h_j$ being omitted. Then $I=\bigcap_{j\in\Lambda^c}I_j$.

    If homogeneous $f\in[I]^\bot\ominus[\sqrt{I}]^\bot$ is orthogonal to $\sum_{j\in\Lambda^c}[I_j]^\bot$, then $f$ belongs to each $I_j$ and therefore $f\in I$, inducing that $f=0$. Then we can find homogeneous $f_j\in[I_j]^\bot\ominus[\sqrt{I}]^\bot$ of the same degree as $f$ and such that $f=\sum_{j\in\Lambda^c}f_j$. Let $\{f^{(n)}\}$ be any sequence of nonzero homogeneous polynomials in $[I]^\bot\ominus[\sqrt{I}]^\bot$, such that $\lim_{n\to\infty}\deg(f_n)=\infty$. Then by the proof of the previous lemma we have
    \begin{eqnarray*}
        &&\liminf_{n\to\infty}\langle\sum_{j\in\Lambda^c}S_{h_j}S_{h_j}^*f^{(n)},f^{(n)}\rangle\cdot||f^{(n)}||^{-2}\\
        &=&\liminf_{n\to\infty}\langle\sum_{j\in\Lambda^c}S_{h_j}S_{h_j}^*f_j^{(n)},f_j^{(n)}\rangle\cdot||f^{(n)}||^{-2}\\
        &\geq&\liminf_{n\to\infty}\sum_{j\in\Lambda^c}(1-|u_j|^2)^2||f_j^{(n)}||^2\cdot||f^{(n)}||^{-2}\\
        &\geq&\min_{j\in\Lambda^c}(1-|u_j|^2)^2>0.
    \end{eqnarray*}
\end{proof}
\begin{cor}\label{cor:in-compact}
    Let $u\in\partial\mathbb{D}^d\backslash\mathbb{T}^d$, and $\Lambda\subset\{1,\ldots,d\}$ such that $u_i\in\mathbb{T}$ if $i\in\Lambda$, and $u_i\in\mathbb{D}$ if $i\in\Lambda^c$. Define polynomials $h(z)=\sum_{i\in\Lambda}\bar{u}_iz_i$ and $h_j(z)=z_j-u_jh(z)$ for $j\in\Lambda^c$. Let $I\subset\poly$ be an ideal of variety $V_u$ which is not essentially quasi-prime, then $\sum_{j\in\Lambda^c}S_{h_j}S_{h_j}^*$ cannot be compact.
\end{cor}
\begin{proof}
    Let $J\subset\poly$ be the ideal generated by $\{h_j:j\in\Lambda^c\}$. For positive integer $n$, denote by $I_n\subset\poly$ the ideal generated by $I'$ and $J^n$. Since $J\subset\sqrt{I'}$, there is a positive integer $N$ such that $J^N\subset I'$, and therefore $I'=I_N$. Since $\dim [I']^\bot\ominus[\sqrt{I}]^\bot=\sqrt{I}/I'=\infty$, for each integer $m>0$ we can find a homogeneous $f\in[I_N]^\bot\ominus[\sqrt{I}]^\bot$ of degree greater than $m+N$. Suppose $f\in[I_{n+1}]^\bot$ but $f\notin[I_n]^\bot$ where $1\leq n<N$. If $n=1$ then $f\in[I_2]^\bot\ominus[\sqrt{I}]^\bot$. If $n>1$ then there is some homogeneous $g_0\in J^{(n-1)}$ of degree $n-1$ such that $S_{g_0}^*f\notin[\sqrt{I}]^\bot$, since otherwise $S_g^*f=0$ for all $g\in J^n$, which contradicts to $f\notin[I_n]^\bot$. Then since $f\in[I_{n+1}]^\bot$, we have $S_g^*S_{g_0}^*f=0$ whenever $g\in J^2$, implying $S_{g_0}^*f\in[I_2]^\bot$. Therefore $S_{g_0}^*f-P_{[I']^\bot}S_{g_0}^*f$ is a non-vanishing element in $[I_2]^\bot\ominus[\sqrt{I}]^\bot$, of degree greater than $m+N-n+1>m$. Hence in either cases there is a sequence $\{f_m\}$ of unit vectors in $[I_2]^\bot\ominus[\sqrt{I}]^\bot$ that converges to $0$ weakly. Then by Corollary \ref{cor:bdb}, $P_{[I']^\bot}\sum_{j\in\Lambda^c}M_{h_j}M_{h_j}^*\mid_{[I']^\bot}$ cannot be compact, proving the conclusion of the corollary.
\end{proof}
Corollary \ref{cor:in-compact} together with Lemma \ref{lem:compactness} indicates that, Theorem \ref{thm:construction} actually gives the necessary and sufficient condition for a homogeneous quotient module to be essentially normal. We summarize these results in the following theorem, which is aforementioned as Theorem \ref{thm:iff}.
\begin{thm}
    Let $I\subset\poly$ be a homogeneous ideal for which $[I]^\bot$ is of infinite dimension. Then $\mathcal{N}=[I]^\bot$ is essentially normal if and only if the following items hold,
    \begin{itemize}
     \item[1)] $Z_{\mathbb{D}^d}(I)$ satisfies Condition A, and hence $Z_{\mathbb D^d}(I)=\bigcup\limits_i V_{u_i}$ for several different discs $V_{u_i}$,
     \item[2)] Let $I=\bigcap_{j=1}^m I_{u_j}$ be the primary decomposition with $Z_{\mathbb D^d}(I_{u_j})=V_{u_j}$, then each $I_{u_j}$ is essentially quasi-prime.
     \end{itemize}
\end{thm}
Finally we give two examples to illustrate how the algebraic structure of the ideal determines the essential normality of its quotient module.
\begin{exam}
    Let $I\subset\mathbb{C}[z_1,z_2,z_3]$ be the ideal generated by $\{(z_1-z_2)^2,z_3(z_1+z_2),z_3^2\}$, then $Z_{\mathbb{D}^3}(I)=\{(z,z,0):z\in\mathbb{D}\}$. Clearly $I'$ is generated by $\{z_1-z_2,z_3z_1,z_3^2\}$, and $\sqrt{I}$ is generated by $\{z_1-z_2,z_3\}$. It is not hard to verify $\sqrt{I}=I'+\mathbb{C}z_3$, and therefore $\dim\sqrt{I}/I'=1$. Then $[I]^\bot$ is essentially normal by Theorem \ref{thm:construction}.
\end{exam}
\begin{exam}
    Let $I\subset\mathbb{C}[z_1,z_2,z_3]$ be the ideal generated by $\{(z_1-z_2)^2,z_3(z_1-z_2),z_3^2\}$, then $Z_{\mathbb{D}^3}(I)=\{(z,z,0):z\in\mathbb{D}\}$. Clearly $I'$ is the ideal generated by $\{z_1-z_2,z_3^2\}$, and $\sqrt{I}$ is the ideal generated by $\{z_1-z_2,z_3\}$. For natural number $n$ the polynomial $e_n(z)=z_3\sum_{i=0}^nz_1^iz_2^{n-i}$ belongs to $\sqrt{I}$, but does not lie in $I'$. Then $\sqrt{I}/I'$ cannot be of finite dimension, and by Theorem \ref{thm:iff} $[I]^\bot$ is not essentially normal.
\end{exam}
Although the forms of the ideals in these examples look similar, the essential normality of their quotient modules are totally different.

\section{K-homology for homogenous quotient module}
Let $I$ be a homogenous ideal such that $[I]^\perp$ is essentially normal, and $\sigma_e([I]^\perp)$ be the joint essential spectrum of the commuting tuple $(S_{z_d},\cdots, S_{z_d}^*)$. Similar to \cite[Proposition 2.5]{Ar3}, it is routine to verify that $C^*([I]^\perp)$ is irreducible.
Therefore if $[I]^\perp$ is essentially normal then $\cal K \subset C^*([I]^\perp)$, and we obtain the following short exact sequence
\begin{equation}\label{eq:extension}
    0\to \cal{K} \hookrightarrow C^*({[I]^\perp})\to C(\sigma_e([I]^\perp))\to 0,
\end{equation}
The $C^*-$algebra extension theory as well as the K-homology is based on this short exact sequence\cite{BDF}.
To investigate the K-homology, we first calculate the essential spectrum $\sigma_e([I]^\perp)$; then examine whether extension (\ref{eq:extension}) yields a non-trivial element. Similar to \cite{WZ}, we have the following lemma.
\begin{lem}
Suppose $u\in\mathbb\partial D^d$ and $I$ be a essentially quasi-prime ideal with variety $V_u$, then $\sigma_e([I]^\perp)=Z(I)\cap\partial \mathbb D^d$.
\end{lem}
\begin{proof} By Theorem \ref{thm:iff}, $\lambda_i-S_{z_i}(i=1,\ldots,d)$ are essentially normal. By \cite[Corollary 3.9]{Cur}, the tuple $(\lambda_1-S_{z_1},\ldots,\lambda_d-S_{z_d})$ is Fredholm if and only if $\sum\limits_{i=1}^{d}(\lambda_i-S_{z_i})(\lambda_i-S_{z_i})^*$ is Fredholm. First we prove the assertion $$Z(I)\cap\partial \mathbb D^d\subset \sigma_e([I]^\perp).$$
Otherwise, there is some $\underline{\lambda}=(\lambda_1,\ldots,\lambda_d)\in Z(I)\cap\partial \mathbb D^d$ making $T=\sum\limits_{i=1}^d (\lambda_i-S_{z_i})(\lambda_i-S_{z_i})^*$ Fredholm. Since $T$ is positive, there is an invertible positive operator $B$ and a compact operator $K$ such that $T=B+K$. Take a sequence $\{\underline{\mu}_n\}$ in $V_u\cap \mathbb D^d$ that converges to $\underline{\lambda}$ as $n\to\infty$. Since $\{k_{\underline{\mu}_n}\}$ converges to $0$ weakly, there is a positive number $c$ such that
\begin{eqnarray*}
\lim\limits_{n\to\infty}\langle T k_{\underline{\mu}_n}, k_{\underline{\mu}_n}\rangle=\lim\limits_{n\to\infty}\langle(B+K)k_{\underline{\mu}_n}, k_{\underline{\mu}_n}\rangle=\lim\limits_{n\to\infty}\langle Bk_{\underline{\mu}_n}, k_{\underline{\mu}_n}\rangle\geq c.
\end{eqnarray*}
However, since $\underline{\mu}_n\in V_u$, $k_{\underline {\mu}_n}\in [I]^\perp$ and it holds that
\begin{eqnarray}
\lim\limits_{n\to\infty}\langle T k_{\underline {\mu}_n},k_{\underline {\mu}_n}\rangle=\lim\limits_{n\to\infty} |\underline{\lambda}-\underline{\mu}_n|^2=0,
\end{eqnarray}
contradicting to the previous inequality. Hence the assertion is proved.

Conversely, $f(S_{z_1},\ldots, S_{z_d})=0$ for each $f\in I$, and then the Spectral Mapping Theorem ensures $\sigma_e([I]^\perp)\subseteq Z(f)$. It follows that $\sigma_e([I]^\perp)\subseteq Z(I).$ Then since $\|S_{z_i}\|\leq 1$ for each $i$, we have $\sigma_e([I]^\perp)\subseteq\overline{\mathbb D}^d$. By Lemma \ref{lem:basic_type}, there is some $i\in\Lambda_u$ such that $S_{z_i}$ is essentially unitary, which implies $\sigma_e(S_{z_i})\subset \mathbb T$. Then for each $(\lambda_1,\cdots, \lambda_d)\in\mathbb D^d$, the tuple $(\lambda_1-S_{z_1},\cdots,\lambda_d-S_{z_d})$ is Fredholm and therefore $\sigma_e([I]^\perp)\subset \partial\mathbb D^d$. The proof of the lemma is completed.
\end{proof}
Now let $I$ be an arbitrary homogenous ideal such that $[I]^\perp$ is essentially normal, then $Z(I)$ satisfies Condition A by Theorem \ref{thm:iff}. Let $Z_{\mathbb D^d}(I)=\bigcup\limits_i V_i$ for different discs $V_i$, and $I=\bigcap_i I_{i}$ be the primary decomposition with $Z(I_i)\cap\mathbb D^d=V_i$. Then by Corollary \ref{cor:ao}, $[I_i]^\perp$ is asymptotically orthogonal to $[I_j]^\perp$ whenever $i\not=j$. It follows that
$$\sigma_e([I]^\perp)=\bigcup_i \sigma_e([I_i]^\perp)=Z(I)\cap \partial\mathbb D^d.$$
We summarize this in the following proposition.
\begin{prop}
Let $I$ be a homogenous ideal such that $[I]^\perp$ is essentially normal, then  $$\sigma_e([I]^\perp)=Z(I)\cap \partial\mathbb D^d.$$
\end{prop}
At the end of this paper, we prove the non-triviality of extension (\ref{eq:extension}).
\begin{thm}
Let $I$ be a homogenous ideal such that $[I]^\perp$ is essentially normal, then the short exact sequence
\begin{eqnarray*}
0\to \cal{K} \hookrightarrow C^*({[I]^\perp})\to C(Z(I)\cap \partial\mathbb D^d)\to 0
\end{eqnarray*}
is not split.
\end{thm}
\begin{proof}
Suppose $Z_{\mathbb{D}^d}(I)=\bigcup_{i}V_{u_i}$ for different discs $\{V_{u_i}\}$. Let $I=\bigcap_i I_{u_i}$ be the primary decomposition such that $V_{u_i}=Z(I_{u_i})\cap\mathbb{D}^d$. Since $I_{u_i}^\prime/I_{u_i}$ is of finite dimension, and $I_{u_i}$ is asymptotically orthogonal to $I_{u_j}$ whenever $i\neq j$, it suffices to show that the short exact sequence
\begin{eqnarray}
0\to \cal{K} \hookrightarrow C^*({[I_{u_j}]^\perp})\to C(\partial V_{u_j})\to 0
\end{eqnarray}
is not split. By \cite[Lemma 5.5]{GWk1}, it is enough to find a Fredholm operator in $C^*({[I_{u_j}]^\perp})$ with nonzero Fredholm index.  Fix a $k\in\Lambda_{u_j}$, then $S_{z_k}^*\mid_{[I_{u_j}]^\perp}$ is essentially unitary. Moreover by Remark \ref{rem:Fredholmindex}, $\Ind(S_{z_k}^*\mid_{[I_{u_j}]^\perp})\not=0$, completing the proof.
\end{proof}

\vskip2.5mm\noindent {\bf Acknowledgements.} The authors would like to thank Yanyue Shi from Ocean University of China, whose critical example of essentially normal quotient module stimulated us to improve the concept of quasi-prime ideals. We thank Professor Kunyu Guo for his encouragements and interests. We also thank Li Chen for his comments on the earlier version of this manuscript.

\end{document}